\documentclass[12pt]{amsart}
\usepackage[colorlinks=true,citecolor=green,linkcolor=magenta]{hyperref}
\usepackage{amsmath}
\usepackage{verbatim}
\usepackage{amsfonts}
\usepackage{amssymb}
\usepackage{blkarray}
\usepackage{color,colortbl}
\usepackage{enumerate}
\usepackage[top=1in, bottom=1in, left=1in, right=1in]{geometry}
\usepackage{mathrsfs}

\usepackage{stmaryrd}
\usepackage{bbold}

\usepackage{cleveref}
\usepackage{soul}

\theoremstyle{definition}
\newtheorem{theorem}{Theorem}[section]

\numberwithin{equation}{section}

\newtheorem{corollary}[theorem]{Corollary}
\newtheorem{lemma}[theorem]{Lemma}
\newtheorem{proposition}[theorem]{Proposition}

\theoremstyle{definition}
\newtheorem{definition}[theorem]{Definition}

\newtheorem{example}[theorem]{Example}

\newtheorem{conjecture}[theorem]{Conjecture}
\newtheorem{remark}[theorem]{Remark}

\newtheoremstyle{TheoremNum}
        {8pt}{8pt}              
        {\upshape}                      
        {}                              
        {\bfseries}                     
        {.}                             
        {.5em}                             
        {\theoremname{#1}\theoremnote{ \bfseries #3}}
  \theoremstyle{TheoremNum}

\newcommand{\m}{\mathfrak{m}}

\newcommand{\Ass}{\operatorname{Ass}}

\newcommand{\sat}{\text{sat}}

\renewcommand{\geq}{\geqslant}

\DeclareMathOperator{\im}{im}

\DeclareMathOperator{\chr}{char}

\title{Expected resurgences and symbolic powers of ideals}

\author[Grifo]{Elo\'isa Grifo}
\address{Department of Mathematics, University of California, Riverside, Riverside, CA 92521, USA}
\email{eloisa.grifo@ucr.edu}

\address{Department of Mathematics, University of Virginia, Charlottesville, VA 22904-4135, USA}

\author[Huneke]{Craig Huneke}
\email{huneke@virginia.edu}

\author[Mukundan]{Vivek Mukundan}
\email{vm6y@virginia.edu}


\begin{document}
\thanks{Huneke was partially supported by
NSF grant DMS-1460638. Grifo was partially supported by an AMS Simons Travel Grant.}

\date{\today}

\begin{abstract} 
We give explicit criteria that imply the resurgence of a self-radical ideal in a regular ring is strictly smaller than its codimension, which in turn implies that the stable version of  Harbourne's conjecture holds for such ideals. One criterion is used to give several explicit families of such ideals, including the defining ideals of space monomial curves. Other results generalize known theorems concerning when the third symbolic power is in the square of an ideal, and a strong resurgence bound for some classes of space monomial curves.
\end{abstract}

\maketitle

\section{Introduction}\label{section intro}

In this paper we build off of recent work by the first author \cite{GrifoStable}, which studied what is called the stable Harbourne
conjecture. A full discussion of this conjecture is contained in that paper; here we highlight some of the relevant facts. The original conjecture of Harbourne states the following:

\begin{conjecture}\cite{Seshadri,HaHu}\label{harbourne} (Harbourne). Let $I$ be a self-radical ideal of big height $c$ in a regular ring $R$.
Then for all $n \geqslant 1$,
$$I^{(cn-c+1)} \subseteq I^n.$$
\end{conjecture}
In the above conjecture, $I^{(n)}$ denotes the $n$-th \textit{symbolic power} of the ideal $I$ which is defined by $I^{(n)}=\bigcap_{P \in \Ass(I)}(I^n R_P \cap R)$. The
{\it big height} of a radical ideal is the largest height (or codimension) of any minimal prime of the ideal. For example, the height or codimension of the ideal $(xy,xz) = (x) \cap (y,z)$ in a polynomial ring in three variables $x,y,z$ is one, while its big height is two. Of course, if the ideal is prime, or more generally equidimensional, then the big height is equal to the height.

The motivation for this conjecture and ensuing work on the comparison of symbolic powers and usual powers comes from the
fact that for regular rings, $I^{(cn)}\subseteq I^n$ for every self-radical ideal of big height $c$ due to \cite{ELS}, \cite{comparison}, and
\cite{MaSchwede}. In fact, 
the value suggested by Harbourne's conjecture is very natural.  Hochster and Huneke's
proof \cite{comparison} that $I^{(cn)}\subseteq I^n$ for regular rings containing a field uses the fact that in prime characteristic p, 
$I^{(cq-c+1)} \subseteq I^{[q]}$ whenever $q = p^e$. In particular, $I^{(cq-c+1)} \subseteq I^{q}$ for all such $q$.
 However,
Conjecture \ref{harbourne} can fail; Dumnicki, Szemberg and Tutaj-Gasi\'nska \cite{counterexamples} found the first
counterexample to $I^{(3)} \nsubseteq I^2$ for certain configurations of points in projective space $\mathbb{P}^2$. Their ideal $I$ is a self-radical ideal of big height $2$. Other examples followed, such as generalizations of the original counterexample \cite{HaSeFermat}, including self-radical ideals of big height $2$ in $\mathbb{P}^N$ for any $N \geqslant 2$, examples over the reals \cite{counterexampleReals}, or certain configurations of lines in $\mathbb{P}^3$ \cite{MalaraSzpond}.

However, many nice classes of ideals verify Harbourne's Conjecture: those defining general points in $\mathbb{P}^2$ \cite{HaHu} and $\mathbb{P}^3$ \cite{Dumnicki2015}, squarefree monomial ideals, or more generally ideals defining F-pure rings in characteristic $p$, or defining rings of dense F-pure type in equicharacteristic $0$ \cite{GrifoHuneke}. This last class contains ideals defining Veronese rings, locally acyclic cluster algebras \cite{ClusterAlgSingularities}, or certain ladder determinantal varieties \cite{GS}, among others. In fact, ideals defining F-pure rings satisfy Harbourne's Conjecture in more generality, over Gorenstein F-finite rings provided that their projective dimension is finite, as shown in work by the first author, Ma, and Schwede \cite{GrifoMaSchwede}.

Moreover, as discussed in \cite{GrifoStable}, there are no counterexamples known to the stable version of Harbourne's conjecture, namely whether
$$I^{(cn-c+1)} \subseteq I^n$$ must hold for all $n$ sufficiently large. In fact, this stable version of Harbourne's Conjecture holds for very general and generic point configurations in $\mathbb{P}^n$ \cite[Theorem 2.2 and Remark 2.3]{StefanYuXie}, and more general sufficient conditions for the stable Harbourne Conjecture are given in \cite{GrifoStable}.  An important invariant in the comparison of symbolic and usual powers
was introduced by Bocci and Harbourne:

\begin{definition}[Bocci-Harbourne \cite{BoH}]
The resurgence of an ideal $I$ is given by 
$$\rho(I)=\sup \left\lbrace \frac{m}{s} ~|~I^{(m)}\not\subseteq I^s \right\rbrace.$$
\end{definition}

Notice that in particular, if $m > \rho(I)\cdot r$, then one is guaranteed that $I^{(m)}\subseteq I^r$. The resurgence of an ideal can be bounded by other invariants \cite[Theorem 1.2.1]{BoH}, and it has been explicitly computed for certain ideals \cite{resurgence2,DHNSST2015,ResurgenceKleinWiman}. Related invariants have also been studied, such as the asymptotic resurgence \cite{AsymptoticResurgence}, which can be computed via integral closures \cite{AsymptoticResurgenceIntClosure}. 

Grifo \cite[Remark 4.17]{mythesis} observed that if the resurgence of $I$ is strictly less than its big height, then the stable Harbourne conjecture is true. Our work in this paper is directed at providing criteria for the resurgence of $I$ to be less than its big height.  We also observe in Proposition \ref{asymptotic resurgence} that even if the potentially weaker condition that the asymptotic resurgence of an ideal is strictly less than its big height, then the stable Harbourne conjecture is true. In Section \ref{criteria} we develop our criterion.  Our main result in this section is Corollary \ref{first corollary}, which states that if the symbolic powers coincide with the saturations with respect to the maximal ideal, and if one can improve the theorem of \cite{ELS,comparison, MaSchwede} to include coefficients in a sufficiently deep power of the maximal ideal, then the resurgence is strictly less than the big height.

In Section \ref{expected} we give some classes of ideals to which the criterion of the previous section apply. One surprising result is that if Harbourne's conjecture
can be verified for a single value $n$ to say that $I^{(cn-c+1)}\subseteq \m I^n$, then the resurgence is less than the big height and consequently the stable Harbourne conjecture is true.

In Section \ref{moncurve} we apply the results of the previous sections to the case of space monomial curves. We prove in general that their resurgence is strictly smaller than
their codimension, and consequently all of them satisfy the stable Harbourne conjecture.  Grifo had previously proved that $P^{(3)}\subseteq P^2$ for the defining ideals for space monomial
curves \cite[Theorem 4.1]{GrifoStable}. In Section \ref{5gen} we generalize this latter fact to ideals defined by the $2\times 2$ minors of a $2\times 3$ matrix in a regular ring such that the ideal of entries of the matrix is
generated by at most $5$ elements. This result is based on work of Seceleanu  \cite[Theorem 3.3]{Seceleanu} as extended by Grifo \cite[Theorem 3.12]{GrifoStable}. 

Section \ref{selflink} computes a strong bound for the resurgence of certain self-linked space monomial curves of $\frac{4}{3}$. Since there are example of such curves with $P^{(4)}\nsubseteq P^3$, this bound
is in general sharp.

\bigskip

\section{Criteria for Expected Resurgence}\label{criteria}

In this section we develop a criterion which forces the resurgence of an ideal to be strictly less than its big height. The main result is Corollary \ref{first corollary}, which we will then apply in the
next sections.

\begin{remark}
Let $I$ be a radical ideal of big height $c$ in a regular ring. Then $1 \leqslant \rho(I) \leqslant c$, where the second inequality follows from \cite{ELS,comparison,MaSchwede}. \end{remark}

In \cite[Cor. 1.1.1]{BoH}, Bocci and Harbourne gave infinite sequences of examples that show that if one takes the supremum of $\rho(I)$ over all homogeneous radical ideals of a
fixed codimension $c$, then in fact that supremum is exactly $c$. However, this does not preclude the very real possibility that no one ideal $I$ can have resurgence equal to its codimension.  We say that the resurgence of $I$ is {\it expected} if it is strictly less than the big height of $I$. 

\begin{remark}
	If $\rho(I)<c$, then $I^{(cn-c+1)}\subseteq I^n$ for $n\gg 0$; more generally, given any $d > 0$, $I^{(cn-d)}\subseteq I^n$ for $n\gg 0$ (cf. \cite[Lemma 4.16]{mythesis} or \cite[Remark 2.7]{GrifoStable}).
\end{remark}

\begin{example}
	The Fermat \cite[Theorem 2.1]{DHNSST2015}, Klein and Wiman \cite[Theorem 1.4]{ResurgenceKleinWiman} configurations all have resurgence $\frac{3}{2}$. In particular, all of these satisfy $I^{(2n-2)} \subseteq I^{n}$ for all $n > 4$. 
\end{example}

In general, explicitly computing the resurgence of a given ideal can be quite difficult, but there are still no examples of ideals of unexpected resurgence. For other classes of ideals whose resurgence has been explicitly computed, see \cite[Theorem 4.11]{GeramitaHarbourneMigliore} or \cite[Theorem 2.4.3]{resurgence2}.

Our first lemma gives what turns out to be a useful result to bound the resurgence away from the codimension.  We prove this lemma in the more general context in which the symbolic powers of an ideal $I$ are cofinal with the usual powers. In this case Swanson \cite{Swanson} proved that the symbolic powers of $I$ are linearly equivalent to the powers of $I$: there exists a constant $s$ such that  $I^{(sn)}\subseteq I^n$ for all $n\geq 1$. We codify this constant in the next definition. 

\smallskip
\begin{definition}
For a radical ideal $I$ in a Noetherian ring $R$, we define the \textit{Swanson constant} of $I$ to be the least integer $s$ such that $I^{(sn)}\subseteq I^n$ for all $n\geq 1$, provided that the symbolic and adic topologies of $I$ are the same. 
\end{definition}

We refer to \cite{Swanson,HKV} for more on the study of Swanson constants. If the ring is regular, \cite{ELS,comparison,MaSchwede} show that $s \leqslant c$, the big height of $I$, and in particular, that the Swanson constant is finite. More generally, as mentioned above, the Swanson constant is finite whenever the the symbolic and adic topologies of $I$ are equivalent \cite{Swanson}.

\begin{lemma}\label{lemma resurgence}
Let $I$ be a radical ideal with Swanson constant $s$.
If $I^{(tsn)}\subseteq I^{rn}$ for some fixed $t,r$ with $t<r$ and all $n \gg 0$, then $\rho(I)\leqslant \frac{ts}{r} < s$.
\end{lemma}

\begin{proof}
Consider the set
$$L = \left\lbrace (m,l) \in \mathbb{Z}^2: I^{(m)} \nsubseteq I^{l} \textrm{ and } m \geqslant l \right\rbrace.$$ 
First, we claim that if $(m,l) \in L$ then $\frac{m}{l} < s$. Indeed, if $m$ and $l$ are positive integers with $\frac{m}{l} \geq s$, then
$$I^{(m)} \subseteq I^{(sl)} \subseteq I^l.$$
Recall that $\rho(I) = \sup \left\lbrace \frac{m}{l} | (m,l) \in L \right\rbrace$. If this supremum is achieved as a maximum, meaning if $\rho(I) = \frac{m}{l}$ for some $(m,l) \in L$, then we have $\rho(I)<s$. If $\rho(I)$ is not computed by an explicit pair $(m,l) \in L$, then we can still choose sequences of positive integers $\left\lbrace m_i \right\rbrace_i$ and $\left\lbrace l_i \right\rbrace_i$ such that $\left\lbrace m_i \right\rbrace_i$ is strictly increasing, $(m_i,l_i) \in L$ for all $i$, and $\frac{m_i}{l_i} \rightarrow \rho(I)$.

	For each $i$, consider integers $n_i, u_i \geqslant 0$ such that $m_i=tsn_i+u_i$ and $u_i<ts$. By assumption,
	\begin{align*}
	I^{(m_i)}\subseteq I^{(tsn_i)}\subseteq I^{rn_i}\text{ for } i \gg 0,
	\end{align*}
	where the last containment follows from the hypothesis.
	
	Since $I^{(m_i)}\not\subseteq I^{l_i}$, we must have $l_i>rn_i$. 
	Therefore,
	\begin{align}
	\frac{m_i}{l_i}&<\frac{m_i}{rn_i}=\frac{tsn_i+u_i}{rn_i}=\frac{ts}{r}+\frac{u_i}{rn_i} <\frac{ts}{r}+\frac{ts}{rn_i}= \frac{st}{r} \left(1+\frac{1}{n_i}\right).
	\end{align}
Note that since $m_i \rightarrow \infty$, we must have $n_i\rightarrow\infty$. As a consequence, we obtain 
\begin{align*}
\lim_{n_i\rightarrow \infty} \frac{st}{r} \left(1+\frac{1}{n_i}\right) = \frac{st}{r} < s.
\end{align*}
Therefore,
$$\rho(I) = \displaystyle\lim_i  \left( \frac{m_i}{l_i} \right) \leqslant \frac{st}{r} < s.$$ 
\end{proof}

\smallskip

To apply Lemma \ref{lemma resurgence}, we first give sufficient conditions for the hypothesis to hold. In what follows, note that over any noetherian ring, $\displaystyle\bigcup_{n \geqslant 1} \Ass(R/I^n$) is a finite set by \cite{BroAssympAssociatedPrimes}. Note also that by \cite{ReesAnalUn} (see also \cite[Theorem 9.1.2]{HunekeSwansonIntegral2006}), if $(R, \m)$ is an analytically ramified Noetherian local ring, then there exists $k$, depending on $I$, such that such that $\overline{I^{n+k}}\subseteq I^n$ for all $n\geq 0$; in fact, these two conditions are equivalent. Under suitable circumstances, such as taking $R$ to be essentially of finite type over an excellent Noetherian local, such $k$ can be taken to uniformly \cite[Theorem 4.13]{HunekeUniform}.

\begin{theorem}\label{first theorem}
	Let $I$ be a self-radical ideal in a analytically unramified Noetherian local ring $(R,\m)$. Let $P_1, \ldots, P_k$ be all the primes that are embedded to $I^n$ for some $n$, and $J = P_1 \cap \cdots \cap P_k$. Assume that there exists an integer $s$ and an $\alpha > 0$ such that
\begin{align}\label{star cond}
I^{(sn)} \subseteq J^{ \lfloor\frac{n}{\alpha}\rfloor} I^n \quad\text{for}\quad n \gg 0.
\end{align}\label{condition}
Then $I^{(tsn)} \subseteq I^{rn}$ for some fixed $t<r$ and all $n\gg 0$.
\end{theorem}

\begin{proof}
We will first prove the following corollary of \cite[Main Theorem]{Swanson97}: under our assumptions, there exists a constant $l$, not depending on $n$, such that $J^{ln} I^{(n)} \subseteq I^n$.

Given $i$ and $n$, if $P_i$ is associated to $I^n$, write $Q_{i,n}$ for a $P_i$-primary ideal appearing in a primary decomposition of $I^n$ such that $P_i^{ln} \subseteq Q_{i,n}$. Such $Q_{i,n}$ exists by \cite[Main Theorem]{Swanson97}. If $P_i$ is not associated to $I^n$, write $Q_{i,n} = R$, so that $P_i^{ln} \subseteq Q_{i,n}$ still holds trivially. Now for each $n$ we have $I^n = I^{(n)} \cap Q_{1,n} \cap \cdots \cap Q_{k,n}$. Moreover,
$$J^{ln} \subseteq P_1^{ln} \cap \cdots \cap P_k^{ln} \subseteq Q_{1,n} \cap \cdots \cap Q_{k,n}.$$

Now the corollary of \cite[Main Theorem]{Swanson97} follows:
	\begin{align}\label{thm1eq1}
	J^{ln} I^{(n)} \subseteq \left( Q_{1,n} \cap \cdots \cap Q_{k,n} \right) I^{(n)} \subseteq Q_{1,n} \cap \cdots Q_{k,n} \cap I^{(n)} = I^n.
	\end{align}
	Also, by \cite{ReesAnalUn}, there exists $k$, depending on $I$,
such that 
	\begin{align}
	\overline{I^{n+k}}\subseteq I^n\text{ for all }n\geq 0.
	\end{align}
Thus it is enough to show that $I^{(tsn)} \subseteq \overline{I^{rn+k}}$ for some fixed $t<r$ and all $n\gg 0$. For all $b>0$, $t \geqslant 2$ and $n \gg 0$, our assumption that $(I^{(sn)}) \subseteq J^{ \lfloor\frac{n}{\alpha}\rfloor} I^n$
 guarantees that
	\begin{align}\label{thm1eq2}
	(I^{(stn)})^{b+1} & \subseteq (J^{\lfloor\frac{tn}{\alpha}\rfloor} I^{nt})^b I^{(stn)} \subseteq J^{b\lfloor\frac{tn}{\alpha}\rfloor} I^{bnt} I^{(stn)}.
	\end{align}

	Fix $b \geqslant 2 l s \alpha$. For $n>0$, let $tn=u\alpha+v$ where $0\leqslant v<\alpha$. Now
	\begin{align*}
	b\left\lfloor\frac{tn}{\alpha}\right\rfloor=bu&>2ls\alpha u\\
	&=ls\alpha u+ls\alpha u\\
	&=ls(tn-v)+ls\alpha u\\
	&=lstn-lsv+ls\alpha u=lstn+ls(\alpha u-v)\\
	&>lstn.
	\end{align*}

Applying \eqref{thm1eq1}, we obtain
$$J^{b \lfloor \frac{tn}{\alpha} \rfloor} I^{(stn)} \subseteq J^{lstn} I^{(stn)} \subseteq I^{stn}.$$ 
Combining with \eqref{thm1eq2}, we obtain
	\begin{align}\label{thm1eq3}
	((I^{(stn)})^{b+1} & \subseteq I^{bnt+snt}
	\end{align}
	for any $t\geq 2$ and $n\gg0$.
Now choose an integer $r \geqslant 3$ such that $r \geq \frac{1+\gamma k}{1-\gamma}$, where $\gamma=\frac{b+1}{b+s}$. Thus
\begin{align*}
r-1\geq (r+k)\gamma.
\end{align*}
Pick any integer $t$ such that 
\begin{align*}
r > t\geq (r+k)\gamma.
\end{align*}
For all $n\geq 1$, we have
\begin{align}\label{eq3}
t & \geq \left(r+\frac{k}{n}\right)\gamma=\left(r+\frac{k}{n}\right)\left(\frac{b+1}{b+s}\right).
\end{align}
Equation \ref{eq3} can also be reduced to
\begin{align}
bnt+snt=(b+s)tn&\geq (rn+k)(b+1).
\end{align}
By \eqref{thm1eq3}, we have
$$\left( I^{(stn)} \right)^{b+1} \subseteq I^{bnt+snt} \subseteq (I^{rn+k})^{b+1}.$$
This shows that $I^{(stn)} \subseteq \overline{I^{rn+k}}$, as desired.
\end{proof}

\begin{remark}
	In the proof of the above theorem, the choice of $b$ can be sharpened to satisfy $b\geq ls\alpha+1$. But the reasoning takes a couple of steps longer and does not make any significant improvements to the proof.
\end{remark}

\begin{corollary}\label{first corollary}
	Let $I$ be an ideal in a regular local ring $(R,\m)$ of dimension $d$, and let $I$ have big height $c \geqslant 2$. Let $J$ be the intersection of all the embedded primes of $I^n$ for some $n$. If there exists an $\alpha > 0$ such that
	$$I^{(cn)} \subseteq J^{ \lfloor\frac{n}{\alpha}\rfloor}I^n\text{ for } n \gg 0.$$
	Then $\rho(I) < c$. 
	
	In particular, if $I^{(n)}_p=I^n_p$ for every prime ideal $p\neq\m$ and there exists an $\alpha > 0$ such that
	$$I^{(cn)} \subseteq \m^{ \lfloor\frac{n}{\alpha}\rfloor}I^n\text{ for } n \gg 0,$$
	then $\rho(I)<c$.
\end{corollary}

\begin{proof}
	By Theorem \ref{first theorem}, there exists integers $r,t$ with $t<r$ and $I^{(ctn)} \subseteq I^{rn}$. Lemma \ref{lemma resurgence} finishes the proof. If $I^{(n)}_p=I^n_p$ for every prime ideal $p\neq\m$, we have $I^{(cn) }=(I^{cn})^\sat$, and we can take $J = \m$ in Theorem \ref{first theorem}.
\end{proof}

\medskip

\begin{remark} 
Corollary \ref{first corollary} shows that one of the conjectures made in \cite{HaHu} implies the stable Harbourne property. Namely, in \cite[Conjecture 2.1]{HaHu} the authors conjectured that if $I$ is the ideal of a finite set of points in projective $N$-space, and $\m$ is the homogeneous maximal ideal, then $I^{(rN)} \subseteq \m^{r(N-1)}I^r$ holds for all $r\geq 1$. Such ideals have codimension $N$ and clearly satisfy the hypothesis of the corollary.  Thus, if this conjecture is true, all such ideals satisfy the stable Harbourne conjecture. \end{remark}

\medskip

\begin{remark} 
The condition in Lemma \ref{lemma resurgence} is closely related to the asymptotic resurgence, first defined in \cite{AsymptoticResurgence}. We recall that this invariant is defined as follows:
The {\it asymptotic resurgence} of an ideal $I$ is given by 
$$\hat{\rho}(I)=\sup \left\lbrace \frac{m}{s} ~|~I^{(mt)}\not\subseteq I^{st} \, \text{for all} \, t \gg 0\right\rbrace.$$
In our next Proposition we prove that if the asymptotic resurgence of an ideal is strictly less than its big height then the stable Harbourne conjecture holds, generalizing Grifo's observation \cite[Lemma 4.16]{mythesis}. Potentially, this invariant could be used to improve some of the results in our paper. However, we found that using the resurgence
gives us stronger results using the methods in this paper.
\end{remark}

\medskip

\begin{proposition}\label{asymptotic resurgence}
Let $I$ be a radical ideal in either a regular local ring $(R,\m)$ containing a field, or a quasi-homogeneous radical ideal in a polynomial ring over a field with irrelevant maximal ideal $\m$. Let $c \geqslant 2$ be the big height of $I$. If $\hat{\rho}(I) < c$, then 
$$I^{(cN-A)}\subseteq I^N$$ 
for any positive integer $A$ and all $N \gg 0$. In particular,
$$I^{(cN-c+1)}\subseteq I^N$$ 
for all $N \gg 0$. 
	\end{proposition} 
	
\begin{proof} Set $d$ equal to the dimension of $R$ and $D = dc+A$. Consider any $n > \frac{D}{c-\hat{\rho}(I)}$. It follows that $\frac{cn-D}{n} > \hat{\rho}(I)$. By definition of asymptotic resurgence, this means that for infinitely many $t$, we must have that $I^{((cn-D)t)} \subseteq I^{nt}$. In particular, $(I^{(cn-D)})^t\subseteq (I^n)^t$, from which it follows that  $I^{(cn-D)}$ is in the integral closure of $I^n$. By the Brian\c con-Skoda theorem, we obtain that $I^{(cn-D)} \subseteq I^{n-d}$. But then $I^{(cn-dc-A)}\subseteq I^{n-d}$. Replacing $n-d$ by $N$, we have for all large $N$ that $I^{(cN-A)}\subseteq I^N$, as desired. 
\end{proof}

\bigskip

\section{Ideals with Expected Resurgence}\label{expected}

In this section we apply the various criteria of the previous section to give classes of ideals with expected resurgence. All such ideals will satisfy the stable Harbourne conjecture, as shown by Grifo \cite[Lemma 4.16]{mythesis}. We chiefly use Corollary \ref{first corollary}. In the homogeneous case, it turns that if the degrees of the generators are small compared to the codimension, then the condition in the corollary is satisfied.

\begin{theorem}
	Let $k$ be a field and $R=k[x_1,\dots,x_d]$ be a polynomial ring of dimension $d$ and $\chr k=0$. Let $I$ be an ideal of big height $c \geqslant 2$ such that $I_P$ is a complete intersection for
	primes $P\ne \m = (x_1,...,x_d)$, and such that $I$ is generated by forms of degree $a<c$. Then $\rho(I)<c$. In particular, all such ideals satisfy the stable Harbourne conjecture.
\end{theorem}

\begin{proof}
	 By repeated use of Euler's formula, we can show that $I^{(cn)}\subseteq \m I^{(cn-1)}\subseteq \cdots \subseteq \m^{cn-1}I$. Since $I$ is generated by forms of degree $a$, we have $\m^{cn-1} I\subseteq \m^{cn-1+a}$. Thus 
	\begin{align*}
	I^{(cn)}\subseteq \m^{cn-1+a}\cap I^n\subseteq I^n\m^{n(c-a)+a-1}\subseteq \m^nI^n.
	\end{align*}
	We can now apply Corollary \ref{first corollary} with $\alpha = 1$ to finish the proof.
\end{proof}
For squarefree monomial ideals $I$, \cite[Corollary 3.6]{TaiTrung2018} or \cite[Theorem 3.17]{AsymptoticResurgenceIntClosure} shows that the resurgence is bounded above by the maximal degree of a minimal generating set of the ideal $I$.

Recall that if the resurgence is expected, then the stable version of Harbourne's conjecture must hold. It is natural to ask \cite[Conjecture 4.2]{mythesis} if Harbourne's conjecture holds for a single value, does it hold in general? Remarkably, the next theorems come close to this statement:

\smallskip

\begin{theorem}\label{theorem if one m then many m}
	Let $I$ be a radical ideal of big height $c \geqslant 2$ in either a regular local ring $(R,\m)$ containing a field, or a quasi-homogeneous radical ideal of big height $c\geqslant 2$ in a polynomial ring over a field with
	irrelevant maximal ideal $\m$. If 
	\begin{align*}
	I^{(ct-c+1)}\subseteq\m I^t \text{ for some fixed }t,
	\end{align*}
	then
	\begin{align*}
	I^{(c(tq+r))}\subseteq \m^q I^{tq+r}
	\end{align*}
	for all $q \geqslant 1$ and $r \geqslant 0$.
	In other words, $I^{(cn)}\subseteq \m^{\left\lfloor \frac{n}{t} \right\rfloor}I^n$.
\end{theorem}

\begin{proof}
	Note that
	$$c(tq+r) = ctq+cr = c(q+r) + (ct-c)q.$$
By \cite[Theorem 4.3 (1)]{MJohnson} with $n=q+r, s_1=\cdots=s_{r}=0$, and $s_{r+1}=\cdots=s_{r+q}=ct-c$,
$$I^{(c(tq+r))} = I^{(c(q+r)+(ct-c)q)} \subseteq I^r(I^{(ct -c+1)})^q.$$
By assumption, $I^{(ct - c+1)} \subseteq \m I^t$, and thus
	$$I^{(c(tq+r))} \subseteq I^r(I^{(ct- c+1)})^q \subseteq I^r (\m I^t)^q = \m^q I^{tq+r}.$$
\end{proof}

\begin{theorem}\label{main resurgence} Let $I$ be a radical ideal of big height $c \geqslant 2$ in a regular local ring $(R,\m)$ containing a field, or a quasi-homogeneous radical ideal of big height $c\geqslant 2$ in a polynomial ring over a field with
	irrelevant maximal ideal $\m$.
	 If 
	\begin{align*}
	I^{(ct-c+1)}\subseteq\m I^t \text{ for some fixed }t,
	\end{align*}
	and if $I_p$ has the property that $I_p^{(n)} = I_p^n$ for all $p\ne \m$ and for all $n$, then
	$\rho(I) < c$.
\end{theorem}

\begin{proof}The hypothesis give that the symbolic powers of $I$ coincide with the saturations of the powers of $I$.  We now can combine
Theorem \ref{theorem if one m then many m} with Corollary \ref{first corollary} to give the result.
\end{proof}

\bigskip

\section{Space monomial curves}\label{moncurve}
A rich source of examples concerning the behavior of symbolic powers are the monomial space curves:

Let $k$ be a field, $R = k[x,y,z]$, and consider the the ideal defining $P$ of the monomial curve $k[t^a, t^b, t^c]$. More precisely, $P$ is the kernel of the map sending $x \mapsto t^a$, $y \mapsto t^b$, and $z \mapsto t^c$. By \cite{Herzog1970}, $P$ is generated by the $2 \times 2$ minors of a matrix of the form
\begin{align}\label{monomialcurvePresMat}
M = \begin{pmatrix}
x^{a_1} & y^{b_1} & z^{c_1} \\
z^{c_2} & x^{a_2} & y^{b_2}
\end{pmatrix},
\end{align}
	which we will write as $I = (F, G, H)$, where
	\begin{align*}
		F = & \, y^{b_1 + b_2} - x^{a_2} z^{c_1} \\
		G = & \, z^{c_1 + c_2} - x^{a_2} y^{b_2} \\
		H = & \, x^{a_1 + a_2} - y^{b_1} z^{c_2}.
	\end{align*}
	
Due in part to this explicit presentation of these ideals, calculations are feasible. However, the symbolic powers of such ideals are still quite mysterious. There are even
examples in which the symbolic power algebra $\oplus P^{(n)}$ is not Noetherian. For example, Goto, Nishida, and Watanabe \cite{SymbReesNotNoeth} showed that for $n \geq  4$ the defining ideal of $k[t^{7n-3}, t^{(5n-2)n}, t^{8n-3}]$ does not have a finitely generated symbolic Rees algebra if $k$ is a field of characteristic $0$. If $k$ has characteristic $p$ and none of $a,b,c$ are squares, then Cutkosky \cite[Theorem 1]{Cutkosky} characterized when the symbolic power Rees algebra is Noetherian in terms of the existence of elements of certain degrees in specificed symbolic powers. Remarkably, even the stable Harbourne conjecture is not known for these primes, although Grifo \cite{GrifoStable} proved that $P^{(3)} \subseteq P^2$ for every prime $P$ defining a monomial curve.  In this section we use a fairly explicit description of the third symbolic power to prove that in fact
$P^{(3)} \subseteq \m P^2$ for all such $P$, and then use our work in the previous sections to prove that the resurgence is strictly less than $2$, the codimension of such primes.  This then implies
the stable Harbourne conjecture holds for space monomial primes. 

\medskip

We begin by observing that
\begin{align}
	a (a_1 + a_2) & = b b_1 + c c_2 \\
	b (b_1 + b_2) & = a a_2 + c c_1 \\
	c (c_1 + c_2) & = a a_2 + b b_2
\end{align}
so that $F,G,H$ are homogeneous when we consider $R$ with the grading given by $\deg x = a$, $\deg y = b$, and $\deg z = c$ generated by the quasi-homogeneous elements $F,G,H$.
In what follows, we will assume that $P$ is not a complete intersection, meaning that $(F, G, H)$ form a minimal generating set for $P$. Moreover, that also implies that $a_1, a_2, b_1, b_2, c_1, c_2 > 0$. As a consequence,  \cite[Corollary 2.5]{Huneke1986} guarantees that $P^{(n)} \neq P^n$ for all $n \geqslant 2$.

By \cite{SchenzelExamples},
	$$P^{(2)} = \left( P^2, \Delta_1 \right),$$
where $\Delta_1$ denotes the determinant of the matrix
	$$D = \begin{pmatrix}
		x^{a_1} & y^{b_1} & z^{c_1} \\
		z^{c_2} & x^{a_2} & y^{b_2} \\
		x^{a_1-a_2+a}y^{b_1+b} z^{c} & x^{a}y^{b_1-b_2+b}z^{c_1+c} & x^{a_1+a}y^{b}z^{c_1-c_2+c}
	\end{pmatrix}$$

	We include  the main results of \cite{ExplicitP3Symbolic} used in this section for the convenience of the reader. Our method to prove that $P^{(3)}\subseteq \m P^2$ for all such $P$  is
	a brute force calculation of the degrees of each generator of the third symbolic power (these ideals are quasi-homogeneous) to show that no such generator can also be a minimal generator of $P^2$. 
	Since Grifo \cite{GrifoStable} proved that $P^{(3)}\subseteq P^2$, this degree calculation suffices.

\begin{theorem}[Theorems 2.1, 2.2 and 2.3 in \cite{ExplicitP3Symbolic}]\label{schenzel formulas}
$\,$

Case 1: Assume $a_1\leqslant a_2,b_1\geq b_2,c_1\geq c_2$ and write
$$\alpha=\max\{0,2a_1-a_2 \},\beta=\max\{0,2b_2-b_1 \},\gamma=\max\{0,2c_2-c_1 \}.$$
	\begin{enumerate}[$(a)$]
		\item If $\alpha=0$, then there exists $\Delta_{21}$ such that 
		\begin{align*}
		x^{2a_1} \Delta_{21}& = x^{a_1} z^{c_1 - 2 c_2 + \gamma} H^3 -  y^{2b_1 - 2 b_2} z^\gamma F G^2 + y^{b_1-b_2} z^{c_1-c_2+ \gamma}G^2 H,\\
		y^{b_2}\Delta_{21}&= -z^{\gamma}F\Delta_1-x^{a_2-a_1}z^{c_1-2c_2+\gamma}GH^2-x^{a_2-2a_1}y^{b_1-b_2}z^{c_1-c_2+\gamma}G^2H,\\
		z^{c_2-\gamma}\Delta_{21}&= H\Delta_1+x^{a_2-a_1}y^{b_1-2b_2}G^3\text{ and }\\
		P^{(3)} &=(P^3,\Delta_1 P,\Delta_{21}).
		\end{align*}
		\item If $\beta=0$, then there is an element $\Delta_{22}$ such that 
		\begin{align*}
		x^{a_1}\Delta_{22} & = z^{c_1-2c_2+\gamma}H^3+y^{b_1-2b_2}F^2G,\\
		y^{b_2}\Delta_{22} &= -z^{\gamma} F\Delta_1-x^{a_2-a_1}z^{c_1-2c_2+\gamma}GH^2,\\
		z^{c_2-\gamma}\Delta_{22} &=H\Delta_1-x^{a_2-a_1}y^{b_1-2b_2}FG^2\text{ and }\\
		P^{(3)} &=(P^3,\Delta_1P,\Delta_{22}).
		\end{align*}
		\item In all other case, i.e., $\alpha>0,\beta>0$,there are elements $\Delta_{231},\Delta_{232}$ such that 
		\begin{align*}
		x^{a_1} \Delta_{231}& =y^{2b_2-b_1}z^{c_1-2c_2+\gamma}h^3+z^{\gamma}F^2G,\\
		y^{b_1-b_2}\Delta_{231} & =-z^{\gamma}F\Delta_1-x^{a_2-a_1}z^{c_1-2c_2+\gamma}GH^2,\\
		z^{c_2-\gamma} &= y^{2b_2-b_1}H\Delta_1-x^{a_2-a_1}FG^2,\\
		x^{a_2}\Delta_{232} &= x^{a_1}z^{c_1-2c_2+\gamma}H^3-y^{2b_1-2b_2}z^{\gamma}FG^2+y^{b_1-b_2}z^{c_1-c_2+\gamma}GH^2,\\
		y^{b_2}\Delta_{232} &=-x^{2a_1-a_2}z^{\gamma}F\Delta_1-y^{b_1-b_2}z^{c_1-c_2+\gamma}G^2H
		-x^{a_1}z^{c_1-2c_2+\gamma}GH^2,\\
		z^{c_2-\gamma}\Delta_{232} &=x^{2a_1-a_2}H\Delta_1+y^{2b_1-2b_2}G^3\text{ and }\\
		P^{(3)} & =(P^3,\Delta_1 P,\Delta_{231},\Delta_{232}).
		\end{align*}
		
	\end{enumerate}
	
	Case 2: Assume $a_1>a_2,b_1>b_2,c_1>c_2$ and let
	$$\alpha=\max\{0,2a_2-a_1 \},\beta=\{0,2b_2-b_2 \},\gamma=\max\{0,2c_2-c_1 \}.$$
	\begin{enumerate}[$(a)$]
		\item Assume $\chr k\neq 2$. Then there are quasi-homogeneous elements $\Delta_{21},\Delta_{22},\Delta_{23}\in P^{(3)}$ such that the following relations hold
		\begin{align*}
		x^{a_2}\Delta_{21} & =-y^{\beta}z^{c_1-2c_2+\gamma}H^3-y^{b_1-2b_2+\beta}z^{\gamma}F^2G,\\
		y^{b_2-\beta}\Delta_{21} &=z^{\gamma}F\Delta_1 +z^{c_1-2c_2+\gamma}GH^2,\\
		z^{c_2-\gamma}\Delta_{21} &=iy^{\beta}H\Delta_1 +y^{b_1-2b_2+\beta}FG^2,\\
		x^{a_2-\alpha}\Delta_{22} &=z^{\gamma}F\Delta_1-z^{c_1-2c_2+\gamma}GH^2,\\
		y^{b_2}\Delta_{22} &= x^{a_1-2a_2+\alpha}z^{\gamma}F^3+x^{\alpha}z^{c_1-2c_2+\gamma}G^2H,\\
		z^{c_2-\gamma}\Delta_{22} &=-x^\alpha G\Delta_1-x^{a_1-2a_2+\alpha}F^2H,\\
		x^{a_2-\alpha}\Delta_{23} &= y^\beta H\Delta_1+y^{b_1-2b_2+\beta}FG^2,\\
		y^{b_2-\beta}\Delta_{23} &=-x^\alpha G\Delta_1+x^{a_1-2a_2+\alpha}F^2H,\\
		z^{c_2}\Delta_{23} &=-x^\alpha y^{b_1-2b_2+\beta}G^3-x^{a_1-2a_2+\alpha}y^\beta FH^2\text{ and }\\
		P^{(3)} &=(P^3,\Delta_1P,\Delta_{21},\Delta_{22},\Delta_{23}).	
		\end{align*}
		\item Assume $\chr k=2$. Then there exists an element $\Delta_2\in P^{(3)}$ such that
		\begin{align*}
		x^{a_2-\alpha}\Delta_2&=\Delta_{21},\\
		y^{b_2-\beta}\Delta_2 &=\Delta_{22},\\
		z^{c_2-\gamma}\Delta_2 &=\Delta_{23}\text{ and }\\
		P^{(3)} &=(P^3,\Delta_1P,\Delta_{2}).
		\end{align*}
	\end{enumerate}
\end{theorem}

The following result uses and improves \cite[Theorem 4.1]{GrifoStable}.

\begin{theorem}\label{theorem 3 in m 2}
	Let $k$ be a field of characteristic not $3$, $R = k[x,y,z]$ and $P$ be the ideal defining the monomial curve $k[t^a, t^b, t^c]$. Consider the maximal ideal $\m = (x,y,z)$ in $R$. Then $P^{(3)} \subseteq \m P^2$.
\end{theorem}

\begin{proof}
	By Theorem \ref{schenzel formulas}, $P^{(3)}/(P^2 + \Delta_1P)$ is generated by  at most three quasi-homogeneous generators, where $\Delta_1$ is the generator of $P^{(2)}$ as above. In order to show that $P^{(3)} \subseteq \m P^{2}$, it is enough to show that these three (or less) generators all have degrees strictly larger than any minimal quasi-homogeneous generator of $P^2$, since $P^{(3)} \subseteq P^2$ by \cite[Theorem 4.1]{GrifoStable}. We will use the relations given by \cite[Theorems (2.1), (2.2) and (2.3)]{ExplicitP3Symbolic} described above to determine the degrees of the new generators to conclude that $P^{(3)} \subseteq \m P^2$. Recall that we consider the grading on $R$ given by $\deg x = a$, $\deg y =b$ and $\deg z = c$.

We will follow the cases in \cite[Theorems (2.1), (2.2) and (2.3)]{ExplicitP3Symbolic}, as in Theorem \ref{schenzel formulas}.

\begin{itemize}
	\item Case 1 a) in \cite[Theorem (2.1)]{ExplicitP3Symbolic}:
	
	There is only one more generator of $P^{(3)}$, $\Delta_{21}$, which satisfies the following condition:
	$$x^{2a_1} \Delta_{21} = x^{a_1} z^{c_1 - 2 c_2 + \gamma} H^3 -  y^{2b_1 - 2 b_2} z^\gamma F G^2 + y^{b_1-b_2} z^{c_1-c_2+ \gamma}G^2 H.$$
	Thus
	\begin{align*}
	\deg \Delta_{21} & = c(c_1-2c_2+ \gamma) + \deg(H^3) - \deg(x^{a_1}) \\
	 & = \deg(H^2) + aa_2 + c(c_1-2c_2+ \gamma) \\ 
	 & > \deg(H^2).
	\end{align*}
	Since by assumption $2a_1 \leqslant a_1 + a_2 = \deg H$, we also have
	\begin{align*}
	\deg \Delta_{21} & \geqslant b(b_1-b_2) + c(c_1-c_2+ \gamma) + \deg(G^2) 	
	 > \deg(G^2).
	\end{align*}
	Finally, note that $\Delta_{21}$ also verifies
	\begin{align*}
	y^{b_2}\Delta_{21}=-z^{\gamma}F\Delta_1-x^{a_2-a_1}z^{c_1-2c_2+\gamma}GH^2-x^{a_2-2a_1}y^{b_1-b_2}z^{c_1-c_2+\gamma}G^2H
	\end{align*}
	and since $\deg \Delta_1 > \deg F + bb_2$, 
	\begin{align*}
		\deg \left( y^{b_2}\Delta_{21} \right) & = \deg \left( z^\gamma F \Delta_1 \right) \\
		& = c \gamma + \deg(F) + \deg(\Delta_1) & \textrm{ so that} \\
		\deg \left( \Delta_{21} \right) & > \deg(F^2).
	\end{align*}
	All $6$ minimal generators of $P^2$ have degree at most $\deg(F^2)$, $\deg(G^2)$ or $\deg(H^2)$, and thus $\Delta_{21} \in \m P^2$. 

	\item Case 1 b) in \cite[Theorem (2.1)]{ExplicitP3Symbolic}:
	There is only one generator to worry about, $\Delta_2 = \Delta_{22}$. Note also that $a_1 \leqslant a_2$, $b_1 \geqslant b_2$, and $c_1 \geqslant c_2$.
	
	Also,
	\begin{align*}
		\deg( x^{a_1} \Delta_2) & \geqslant \deg \left( H^3 \right) \\
		& \textrm{ and since $\deg(H) > \deg(x^{a_1})$,} \\
		\deg( \Delta_2) & > \deg \left( H^2 \right). 
	\end{align*}
	Also,
	\begin{align*}
		\deg( x^{a_1} \Delta_2) & = \deg \left( y^{b_1-2b_2} F^2 G \right) \\
		& \textrm{ and since $\deg(G) > \deg(x^{a_2})$ and by assumption $a_1 \leqslant a_2$,} \\
		\deg( \Delta_2) & > \deg \left( y^{b_1-2b_2} \right) + \deg \left( F^2 \right).
		\end{align*} 
		Finally, $c_1 \geqslant c_2$ by assumption, so that $\deg(F) > \deg(z^{c_2}) \geqslant \deg(z^{c_2})$. Then
		\begin{align*}
		\deg(z^{c_2} \Delta_2) & \geqslant \deg \left( x^{a_2-a_1} y^{b_1-2b_2} FG^2 \right) \\
		& \textrm{ so that} \\
		\deg( \Delta_2) & > \deg \left( G^2 \right).
		\end{align*} 
		
		\item Case c) in \cite[Theorem (2.1)]{ExplicitP3Symbolic}
		
		There are two more generators of $P^{(3)}$ we need to worry about, $\Delta_2$ and $\Delta_3$, and
		\begin{align*}
			\deg \left( x^{a_1} \Delta_2 \right) & \geqslant \deg \left( H^3 \right) \\
			\deg \left( y^{b_1-b_2} \Delta_2 \right) & \geqslant \deg \left( G H^2 \right) \\
			\deg \left( z^{c_1 - c_2} \Delta_2 \right) & \geqslant \deg \left( F G^2 \right)
		\end{align*}
		and since 
		\begin{align*}
			\deg \left( H \right) > \deg \left( x^{a_1} \right) \\
			\deg \left( G \right) > 	\deg \left( y^{b_1-b_2} \right) & &\textrm{since by assumption, } 2b_2 \geqslant b_1 \geqslant b_2, \textrm{ so } b_2 \geqslant b_1-b_2 \\
			\deg \left( F \right) > \deg \left( z^{c_1 - c_2} \right)
		\end{align*}
		we conclude that $\deg \left( \Delta_2 \right) > \deg \left( F^2 \right), \deg \left( G^2 \right), \deg \left( H^2 \right)$.
		
		Similarly,
		\begin{align*}
			\deg \left( x^{a_2} \Delta_3 \right) & \geqslant \deg \left( H^3 \right) \\
			\deg \left( y^{b_2} \Delta_3 \right) & \geqslant \deg \left( G H^2 \right) \\
			\deg \left( z^{c_1 - c_2} \Delta_3 \right) & \geqslant \deg \left( G^3 \right)
		\end{align*}
		and
		\begin{align*}
			\deg \left( H \right) > \deg \left( x^{a_2} \right) \\
			\deg \left( G \right) > 	\deg \left( y^{b_2} \right) \\
			\deg \left( F \right) > \deg \left( z^{c_1 - c_2} \right)
		\end{align*}
		so $\deg \left( \Delta_3 \right) > \deg \left( F^2 \right), \deg \left( G^2 \right), \deg \left( H^2 \right)$.
		
		\item In the remaining cases, $a_1 > a_2$, $b_1 > b_2$, and $c_1 > c_2$, and if the characteristic is not $2$, \cite[Theorem (2.2)]{ExplicitP3Symbolic} says that there are three generators of $P^{(3)}$ we need to worry about,  $\Delta_{21}, \Delta_{22}, \Delta_{23}$, where
	\begin{align*}
			\deg \left( x^{a_2} \Delta_{21} \right) & \geqslant \deg \left( H^3 \right) & \textrm{ and since } \deg(H) > \deg \left( x^{a_2} \right) \textrm{ we get } \deg \left( \Delta_{21} \right) > \deg \left( H^2 \right). \\
			\deg \left( x^{a_2} \Delta_{21} \right) & \geqslant \deg \left( F^2 G\right) & \textrm{ and since } \deg(G) > \deg \left( x^{a_2} \right) \textrm{ we get } \deg \left( \Delta_{21} \right) > \deg \left( F^2 \right).\\
			\deg \left( z^{c_1} \Delta_{21} \right) & > \deg \left( F G^2 \right) & \textrm{ and since }  \deg \left( F \right) > \deg \left( z^{c_1} \right) \textrm{ we get } \deg \left( \Delta_{21} \right) > \deg \left( G^2 \right)\\
			\deg \left( x^{a_2} \Delta_{22} \right) & > \deg \left( G H^2 \right) & \textrm{ and since } \deg \left( G \right) > \deg \left( x^{a_2} \right) \textrm { we get } \deg \left( \Delta_{22} \right) > \deg \left( H^2 \right) \\
			\deg \left( y^{b_2} \Delta_{22} \right) & \geqslant \deg \left( F^3 \right) & \textrm{ and since } \deg \left( F \right) > \deg \left( y^{b_2} \right) \textrm { we get } \deg \left( \Delta_{22} \right) > \deg \left( F^2 \right) \\
			\deg \left( y^{b_2} \Delta_{22} \right) & \geqslant \deg \left( G^2 H \right) & \textrm{ and since } \deg \left( H \right) > \deg \left( y^{b_2} \right) \textrm { we get } \deg \left( \Delta_{22} \right) > \deg \left( G^2 \right) \\
			\deg \left( x^{a_2} \Delta_{23} \right) & \geqslant \deg \left( F G^2 \right) & \textrm{ and since } \deg \left( F \right) > \deg \left( x^{a_2} \right) \textrm { we get } \deg \left(\Delta_{23} \right) > \deg \left( G^2 \right) \\
			\deg \left( y^{b_2} \Delta_{23} \right) & \geqslant \deg \left( F^2 H \right) & \textrm{ and since } \deg \left( H \right) > \deg \left( y^{b_2} \right) \textrm { we get } \deg \left( \Delta_4 \right) > \deg \left( F^2 \right) \\
			\deg \left( z^{c_2} \Delta_{23} \right) & \geqslant \deg \left( F H^2 \right) & \textrm{ and since } \deg \left( F\right) > \deg \left( z^{c_2} \right) \textrm { we get } \deg \left( \Delta_{23} \right) > \deg \left( H^2 \right).
		\end{align*}

		\item Finally, when the characteristic is $2$, the generators $\Delta_{21}, \Delta_{22}, \Delta_{23}$ from above can be replaced by one generator $\Delta_2$, where
		\begin{align*}
		x^{a_2-\alpha}\Delta_2&=\Delta_{21},\\
		y^{b_2-\beta}\Delta_2 &=\Delta_{22},\\
		z^{c_2-\gamma}\Delta_2 &=\Delta_{23}\text{ and }\\
		P^{(3)} &=(P^3,\Delta_1P,\Delta_{2}).
		\end{align*}
		
		If $2a_2 \leqslant a_1$, then $\alpha = 0$ and $\Delta_2 = \Delta_{21} \in \m P^2$. Similarly if $\beta = 0$ or $\gamma = 0$, so we might as well assume that $\alpha = 2a_2 - a_1, \beta = 2b_2-b_1, \gamma = 2c_2 - c_1> 0$. Then
		\begin{align*}
		x^{a_1 - a_2}\Delta_2&=\Delta_{21},\\
		y^{b_1 - b_2}\Delta_2 &=\Delta_{22},\\
		z^{c_1 - c_2}\Delta_2 &=\Delta_{23}\text{ and }\\
		P^{(3)} &=(P^3,\Delta_1P,\Delta_{2}).
		\end{align*}
		Note that
		$$\deg \left( x^{a_2} \Delta_{21} \right) \geqslant \deg \left( H^3 \right)$$
		still holds, and thus
		$$\deg \left( x^{a_1} \Delta_{2} \right) = \deg \left( x^{a_2} \Delta_{21} \right) \geqslant \deg \left( H^3 \right).$$ 
		Since $\deg(H) > \deg \left( x^{a_1} \right)$, we conclude that $\deg \left( \Delta_{2} \right) > \deg \left( H^2 \right)$. Similarly, 		
		$$\deg \left( y^{b_1} \Delta_{2} \right) = \deg \left( y^{b_2} \Delta_{22} \right) \geqslant \deg \left( F^3 \right)$$ 
		Since $\deg(F) > \deg \left( y^{b_1} \right)$, we conclude that $\deg \left( \Delta_{2} \right) > \deg \left( F^2 \right)$. Finally,
		$$\deg \left( z^{c_1} \Delta_{2} \right) = \deg \left( z^{c_2} \Delta_{23} \right) \geqslant \deg \left( G^3 \right),$$
		and since $\deg(G) > \deg \left( z^{c_1} \right)$, we conclude that $\deg \left( \Delta_{2} \right) > \deg \left( G^2 \right)$.
\end{itemize}
\end{proof}

Our main result of this section follows quickly: 

\begin{theorem}
	Let $k$ be a field of characteristic not $3$, $R = k[x,y,z]$ and $P$ be the ideal defining the monomial curve $k[t^a, t^b, t^c]$. Then $\rho(I) < 2$, and in particular $P^{(2n-1)} \subseteq P^n$ for all $n \gg 0$.
\end{theorem}

\begin{proof}
	It is enough to show that $P$ verifies the conditions of Corollary \ref{first corollary}. On the one hand, $P$ is a complete intersection on the punctured spectrum; on the other hand, Theorem \ref{theorem 3 in m 2} gives the base case to apply Theorem \ref{theorem if one m then many m}, which in turn gives the remaining condition we need in order to apply Corollary \ref{first corollary}.
\end{proof}

\section{A sufficient condition for $I^{(3)} \subseteq I^2$ for $3$-generated ideals in dimension $3$}\label{5gen}

In the last section we were able to prove that $I^{(3)} \subseteq \m I^2$ whenever $I$ is the defining ideal of a space monomial curve, which in turn allowed us to prove that
the resurgence of $I$ is strictly less than $2$. All such ideals are defined by the $2 \times 2$ minors of a $2 \times 3$ matrix such that the ideal generated by the entries
is a complete intersection, defined by powers of $x,y$, and $z$.  In this section we generalize this to ideals defined by the  $2 \times 2$ minors of a $2 \times 3$ matrix such that the ideal generated by the entries has at most $5$ generators, although we can only prove that $I^{(3)} \subseteq I^2$. This does provide new evidence that this containment holds for all prime ideals of codimension two in a power series ring over an algebraically closed field.  Our result uses the techniques of Seceleanu as extended by Grifo.

\medskip

\begin{theorem}\label{theorem 5 generated}
	Let $R = k[x,y,z]$ or $k \llbracket x,y,z \rrbracket$, where $k$ is a field of characteristic not $3$. Let $a_1, a_2, a_3, b_1, b_2, b_3 \in R$, and
	$$I = I_2 \begin{pmatrix}
		a_1 & a_2 & a_3 \\ b_1 & b_2 & b_3
	\end{pmatrix}.$$
	If $(a_1, a_2, a_3, b_1, b_2, b_3)$ can be generated by $5$ elements or less, $I^{(3)} \subseteq I^2$.
\end{theorem}

\begin{proof}
	By \cite[Theorem 3.12]{GrifoStable} (which extends \cite[Theorem 3.3]{Seceleanu}), it is enough to show that
	$$v = \begin{bmatrix} 0 \\ 0 \\ a_1b_2-a_2b_1 \end{bmatrix} \in \im A,$$
	where
	$$A = \begin{bmatrix}
	a_1 & a_2 & a_3 & 0 & 0 & 0 & b_1 & b_2 & b_3 & 0 & 0 & 0 \\ 
	0 & a_1 & 0 & a_2 & a_3 & 0 & 0 & b_1 & 0 & b_2 & b_3 & 0 \\ 
	0 & 0 & a_1 & 0 & a_2 & a_3 & 0 & 0 & b_1 & 0 & b_2 & b_3
	\end{bmatrix}.$$
	Without loss of generality, we may assume that $b_3 \in (a_1, a_2, a_3, b_1, b_2)$, and write $b_3 = x_1a_1 + x_2a_2 + x_3a_3 + x_4b_1 + x_5b_2$ for some $x_i \in R$. 
	
	To see that $v \in \im A$, we claim that $v = Aw$, where
	$$w = (
		2  x_1 a_2,
		2 x_2 a_2 + x_3 a_3,
		x_3 a_2 + b_2,
		-(2 x_2 a_1 + 2 x_5 b_1),
		b_1 - x_3 a_1,
		0,
		2 x_4 a_2,
		2 x_5 a_2 - a_3,
		2 a_2,
		0,
		0,
		0).$$
	
	Indeed,
$$\begin{bmatrix} 0 \\ 0 \\ a_1b_2-a_2b_1 \end{bmatrix} =
2  x_1 a_2 \begin{bmatrix} a_1 \\ 0 \\ 0 \end{bmatrix}
+ (2 x_2 a_2 + x_3 a_3) \begin{bmatrix} a_2 \\ a_1 \\ 0 \end{bmatrix}
+ (x_3 a_2 + b_2) \begin{bmatrix} a_3 \\ 0 \\ a_1 \end{bmatrix}
-(2 x_2 a_1 + 2 x_5 b_1) \begin{bmatrix} 0 \\ a_2 \\ 0 \end{bmatrix}$$
$$+ (b_1 - x_3 a_1) \begin{bmatrix} 0 \\ a_3 \\ a_2 \end{bmatrix}
+ 2 x_4 a_2 \begin{bmatrix} b_1 \\ 0 \\ 0 \end{bmatrix} 
+ (2 x_5 a_2 - a_3) \begin{bmatrix} b_2 \\ b_1 \\ 0 \end{bmatrix} 
-2 a_2 \begin{bmatrix} x_1a_1 + x_2a_2 + x_3a_3 + x_4b_1 + x_5b_2 \\ 0 \\ b_1 \end{bmatrix}.$$
\end{proof}

\begin{remark}
	Theorem \ref{theorem 5 generated} generalizes \cite[Theorem 4.1]{GrifoStable}.
\end{remark}

However, given an ideal $I$ generated by the maximal minors of a $2 \times 3$ matrix $M$, the condition $I^{(3)} \subseteq I^2$ does not depend only on $I_1(M)$. 

\begin{example}[Fermat configurations]
By \cite{counterexamples}, $I^{(3)} \nsubseteq I^2$ when $I = I_2(M)$ for
$$M = \begin{bmatrix} x^2 & y^2 & z^3 \\ yz & xz & xy \end{bmatrix}.$$
On the other hand, the condition that $I_1(M)$ is at most $5$-generated is sufficient but not necessary for $I^{(3)} \subseteq I^2$. In fact, we can simply reorder the entries in $M$ above to obtain a matrix $N$ such that $J = I_2(N)$ does verify $J^{(3)} \subseteq J^2$:
$$J = I_2 \begin{bmatrix} x^2 & xz & z^3 \\ yz & y^2 & xy \end{bmatrix}.$$
In fact, in this case $f_3 = x^2yz-y^2x^3$ is a multiple of $b_3 = xy$, and the containment $I^{(3)} \subseteq I^2$ is an easy application of \cite[Theorem 3.12]{GrifoStable}.
\end{example}

\bigskip

\section{Self-Linked Monomial Curves whose symbolic Rees algebras are generated in degree two}\label{selflink}

In this section we focus on a special class of monomial curves, proving that their resurgence is not only strictly less than 2, but proving in fact that the resurgence is at most $\frac{4}{3}$. We also give an example in which the resurgence is at least $\frac{4}{3}$, proving that
this value is sharp in general. The classes of ideals verifying these bounds must also verify Harbourne's Conjecture.

\medskip

\begin{remark}\label{remark resurgence 5/3}
Say we wanted to prove Harbourne's Conjecture \ref{harbourne} holds for an ideal $I$ of big height $h$ simply by studying its resurgence. We want to show that $I^{(hn-h+1)} \subseteq I^n$ for all $n \geqslant 2$. By definition of resurgence, the containment $I^{(hn-h+1)} \subseteq I^n$ must hold as long as
$$\frac{hn-h+1}{n} > \rho(I).$$
Given $\rho(I)$ and $h$, we can then guarantee $I^{(hn-h+1)} \subseteq I^n$ for a particular value of $n$ if
$$n > \frac{h-1}{h- \rho(I)}.$$
Harbourne's Conjecture asks that this containment holds for all $n \geqslant 2$, which would then follow immediately if we knew that
$$2 > \frac{h-1}{h- \rho(I)} \quad \textrm{ or equivalently } \quad \rho(I) < \frac{h+1}{2}.$$
This shows that if $I$ is a self-radical ideal of big height $h$ and $\rho(I) < \frac{h+1}{2}$, then $I$ verifies Harburne's Conjecture.

For the case of space monomial curves in characteristics other than $3$, we already have $I^{(3)} \subseteq I^2$ by \cite[Theorem 4.1]{GrifoStable}, so we only need to prove $I^{(2n-1)} \subseteq I^n$ holds for $n \geqslant 3$. Assume that
$$3 >  \frac{2-1}{2- \rho(I)} \quad \textrm{ or equivalently } \quad \rho(I) < \frac{5}{3}.$$
Then our computations above show that $\frac{2n-1}{n} > \rho(I)$ for all $n \geqslant 3$, and $I$ verifies Harbourne's Conjecture.
\end{remark}

\begin{theorem}
Suppose $R$ be a three dimensional regular local ring and $P$ be a height two ideal. Suppose that:
\begin{enumerate}[$(1)$]
	\item the symbolic Rees algebra of $P$ is generated in degree 2, and
	\item there exists an ideal $J$ such that
	\begin{enumerate}
		\item $P^{(2)} \subseteq JP$ and
		\item $JP^{(2)} \subseteq P^2$.
	\end{enumerate}
\end{enumerate}
Then $\rho(P) \leqslant \frac{4}{3}$.
\end{theorem}

\begin{proof}
Suppose by way of contradiction that $\rho(P) > \frac{4}{3}$. By definition this implies that there exists integers $m,s$ such that $P^{(m)}$ is not contained in $P^s$, and such that
$\frac{m}{s} > \frac{4}{3}$, i.e.,  $3m > 4s$.  We write $m = 4n+i$, where $0 \leqslant i \leqslant 3$.  We claim that if $ i = 0$, then $P^{(m)}\subseteq P^{3n}$. Indeed,
	
$$
	P^{(m)} = P^{(4n)}=(P^{(2)})^{2n}=(P^{(2)})^{n}(P^{(2)})^n \subseteq J^{n}P^{n}(P^{(2)})^n \subseteq P^{n}(JP^{(2)})^n \subseteq P^{3n}.
$$
	
	It follows that $s\geq 3n+1$. Hence if $i\leqslant 1$, then $\frac{m}{s}\leqslant \frac{4n+1}{3n+1}\leqslant \frac{4}{3}$. Therefore, $i\geq 2$. 
	
	Consider the case in which $i = 2$. Then
	
	    \begin{align*}
	P^{(m)}&=P^{(4n+2)}=P^{(2(2n+1))}=(P^{(2)})^{2n+1}\\
	&=(P^{(2)})^{n}(P^{(2)})^{n+1}=(JP)^{n}(P^{(2)})^{n}P^{(2)}\\
	&\subseteq P^{n}P^{2n}P^{(2)}\subseteq P^{3n+1}.
	\end{align*}
	
	It follows that $s\geq 3n+2$. Therefore if $2\leqslant i\leqslant 3$,  $\frac{m}{s} \leqslant \frac{4n+2}{3n+2}< \frac{4}{3}$. Thus $i = 3$. 
	
	Consider the case in which $i = 3$. We have that 
	
	\begin{align*}
         P^{(m)}&=P^{(4n+3)}=(P^{(2)})^{2n+1}P\\
	&=(P^{(2)})^{n}(P^{(2)})^{n+1}P\subseteq (JP)^{n}(P^{(2)})^{n+1}P=(JP^{(2)})^nP^{n+1}P^{(2)}\\
	&\subseteq (P^{2})^nP^{n+1}P^{(2)}\subseteq P^{3n+2}.
	\end{align*}
	
	It follows that $s\geq 3n+3$, and therefore $\frac{m}{s} \leqslant \frac{4n+3}{3n+3}< \frac{4}{3}$.
\end{proof}

\begin{corollary}\label{selflinkedMoncruveswithresurgence4/3}
Suppose $P$ is the defining ideal of the monomial curve $k[t^a,t^b,t^c]$ in the polynomial ring $k[x,y,z]$. If $P$ is a self linked ideal whose presentation matrix is as in \eqref{monomialcurvePresMat} satisfy $b_1=b_2$ or $a_1=a_2,c_1=c_2$, then $\rho(P) \leqslant \frac{4}{3}$. In particular, $I$ verifies Harbourne's Conjecture.
\end{corollary}

\begin{proof}
Let  $\varphi$ is a presentation matrix of $P$. Condition (1) in the previous theorem is satisfied due to \cite[Corollary 2.12]{MonCurvesGen2} (see also \cite[Corollary 4.3]{NoethSymbReesAlgDegrees}). Condition (2)(a) is satisfied due to \cite[Corollary 5]{EisenbudMazur} and condition (2)(b) is satisfied due to \cite[Corollary 2.10]{Huneke1986} with $J=I_1(\varphi)$. The conclusion now follows from the previous theorem.

Finally, $I$ satisfies Harbourne's conjecture by Remark \ref{remark resurgence 5/3}.
\end{proof}

\begin{example}
	Let $P$ be the defining ideal of the monomial curve $k[t^{18},t^{19},t^{231}]$ in the ring $k[x,y,z]$, $k$ a field. Consider a presentation matrix 
	\begin{align*}
	\varphi=\begin{bmatrix}
	x^3 & y & z^{12}\\
	z^3 & x^{13} & y
	\end{bmatrix}
	\end{align*}
	of $P$. This presentation shows that the ideal $P$  is self-linked \cite[Theorem 1.1]{MonCurvesGen2}. Macaulay2 computations show that $P^{(4)}\not\subseteq P^3$. This shows that the bound in Corollary \ref{selflinkedMoncruveswithresurgence4/3} is sharp.
\end{example}

\section*{Acknowledgments}

We thank Alexandra Seceleanu for finding some typos in an earlier version of the paper, and the anonymous referee for their comments.

\bibliographystyle{alpha}
\bibliography{References}
\end{document}